\documentclass{amsart}
\usepackage[utf8]{inputenc}
\usepackage{ifxetex,ifluatex}
\usepackage{adjustbox}
\usepackage{colonequals}
\usepackage{amsaddr}
\usepackage{float}
\usepackage{tikz-cd}
\usepackage{amsmath}
\usepackage{amsfonts}
\usepackage{amsthm}
\usepackage{amssymb}
\usepackage{graphicx}
\usepackage{enumerate}
\usepackage{hyperref}
\usepackage[all,cmtip]{xy}
\usepackage{afterpage}
\usepackage{geometry}
 \usepackage{comment}
\usepackage{cleveref}
\crefformat{section}{§#2#1#3} 
\usepackage{csquotes}
\usepackage{stmaryrd}
\usepackage{lscape} 
\usepackage{pdflscape} 
\usepackage{rotating}
\usepackage{tabularx}
\usepackage{xcolor}
\theoremstyle{plain}
\newtheorem{theorem}{Theorem}[section]
\newtheorem{lemma}[theorem]{Lemma}
\newtheorem{corollary}[theorem]{Corollary}


\theoremstyle{definition}
\newtheorem{example}[theorem]{Example}
\newtheorem{definition}[theorem]{Definition}

\theoremstyle{remark}

\newcommand{\N}{\mathbb{N}}

\newcommand{\F}{\mathbb{F}}

\newcommand{\R}{\mathbb{R}}

\newcommand{\C}{\mathbb{C}}

\newcommand{\J}{\mathcal{J}}

\newcommand{\sqrtf}{\sqrt[{\leftroot{8}\uproot{-10}\scriptstyle \varepsilon}]}

\newcommand{\agmf}{AGM_{\F_q}}
\newcommand{\agmel}{AGM_{\F_{11}}}

\newcommand{\jf}{\J_{\F_q}}

\usepackage{xcolor}

\usepackage{xpatch}

\makeatletter
\xpatchcmd{%
\@maketitle}{%
\ifx\@empty\authors \else \@setauthors \fi
}{%
  \ifx\@empty\authors \else \@setauthors \fi
  \ifx\@empty\addresses \else\@setaddresses\fi
}{\typeout{Patch successful}}{\typeout{Patch failed}}
\makeatother

\begin{document}
\title[AGM over $\F_q$, where $q\equiv5\pmod{8}$]{Arithmetic-geometric mean sequences over finite fields $\F_q$, where $q\equiv5\pmod{8}$}

\author{Nat\'{a}lia B\'{a}torov\'{a}, Stevan Gajovi\'c}

\maketitle
\begin{abstract}
Arithmetic-geometric mean sequences were already studied over real and complex numbers, and recently, Michael J. Griffin, Ken Ono, Neelam Saikia and Wei-Lun Tsai considered them over finite fields $\F_q$ such that $q \equiv 3 \pmod 4$. In this paper, we extend the definition of arithmetic-geometric mean sequences over $\F_q$ such that $q \equiv 5 \pmod 8$. We explain the connection of these sequences with graphs and show the properties of the corresponding graphs in the case $q \equiv 5 \pmod 8$. 
\end{abstract}

\section{Introduction.}
The arithmetic-geometric mean sequence (AGM) is a sequence of ordered pairs where the first element of the pair is the arithmetic mean of the previous pair, and the second element of the pair is the geometric mean of the previous pair. AGM over positive real numbers was first discovered by Lagrange in 1785 and rediscovered by Gauss a few years later, in 1791. AGM sequences over real or complex numbers have many applications.

AGM sequences can also be considered over finite fields. For example, Michael J. Griffin, Ken Ono, Neelam Saikia and Wei-Lun Tsai \cite{griffin2022agm} introduced AGM over finite fields with $q$ elements such that $q \equiv 3 \pmod 4$. In this case, they made a natural definition of infinite AGM for $q >3$.

However, it was not clear if we could define an infinite sequence over every finite field of odd characteristic. In this paper, we extend the definition over finite fields with $q$ elements such that $q \geq 29$ and $q \equiv 5 \pmod 8$. 

This paper is based on the bachelor thesis \cite{Naty-thesis} of the first author, cosupervised by the second author. Very recently, June Kayath, Connor Lane, Ben Neifeld, Tianyu Ni, Hui Xue, published a preprint \cite{NewAGM} in which they generalise results from \cite{griffin2022agm} and define and explain AGM sequences over all finite fields of odd characteristic. To do so, they used elliptic curves. On the other hand, our approach for $q\equiv 5\pmod{8}$ was based on elementary number theory. 

\subsection{Main results}\label{subsec:main-results}

The main results of this paper were entirely new to the authors, even though there is now a preprint \cite{NewAGM} dealing with even more general cases. We define the AGM sequences over $\F_q$ where $q\equiv 5\pmod{8}$, see Definition~\ref{def:AGMq-5-mod8} and describe their corresponding graphs. 

First, we prove that each vertex of that graph has 0 or 2 incoming and outcoming edges by Lemma~\ref{lem:0-2-children} and Corollaries~\ref{cor-two-children-one-in-cycle} and~\ref{cor-two-parents-one-in-cycle}. 
We further prove that a nontrivial connected component of the graph, i.e., which has at least one edge, always contains a cycle, see Theorem~\ref{theoremNonemptyEdges}. This is the property we want from the AGM sequences over finite fields. 

We fully describe the corresponding graph in Theorem~\ref{theoremGraph5mod8} and present a few sketches of such graphs; see Figures~2 and~3. Finally, we prove that for all $q\geq 29$, $q\equiv 5\pmod{8}$, there is at least one cycle in the corresponding graph, see Theorem~\ref{thm-nontrivial-graph-29}.  In addition, we prove slightly more generally that there is at least one connected component, all of whose vertices are squares (both components of the vertex are squares, see Definition~\ref{def-square/non-square-vertices}); see Corollary~\ref{cor-all-squares-component}.

However, we also slightly improve and generalise the results from \cite{griffin2022agm} in Theorems~\ref{theoremNumberOfComponents} and~\ref{theoremSquareandNonsquareVertices}. More precisely, we analyse components of the corresponding graphs more precisely and give a little better divisibility conclusion that in \cite[Theorem 1(4)]{griffin2022agm}.

\subsection{Structure of the paper}

For the convenience of readers, we start with briefly discussing various applications of the AGM sequences in Section~\ref{sec:applications}.

In Section~\ref{secAGMGraph}, we introduce the AGM sequences over finite fields and their corresponding graphs. We always use Definition~\ref{defJfq} for the corresponding graph of the AGM sequence. We mostly follow \cite{griffin2022agm}, but we include our improved results there. We summarise the results from \cite{griffin2022agm} in Section~\ref{chap3mod4}. 

Section~\ref{chap5mod8} contains our main results, which were introduced in \cref{subsec:main-results}.

In Section~\ref{sec-challenge}, we explain the strategy used to switch from $q\equiv 3\pmod{4}$ to $q\equiv 5\pmod{8}$ and why this strategy does not work for $q\equiv 1\pmod{8}$. We recall that so far, one needs to use elliptic curves for this case, as in \cite{NewAGM}. We finish this paper with a challenge for interested readers to try to discover an elementary solution to extend the AGM sequences for $q\equiv 1\pmod{8}$.

\subsection*{Acknowledgements} We warmly thank V\'{i}t\v{e}zslav Kala for many valuable comments and suggestions throughout this work. We also thank Steffen M\"{u}ller, Lazar Radi\v{c}evi\'c, Magdal\'{e}na Tinkov\'{a}, Jaap Top, and Pavlo Yatsyna for helpful discussions. We thank the anonymous referee for several useful comments that improved the paper. S.G. was supported by the Czech Science Foundation GAČR, grant 21-00420M and a Junior Fund grant for postdoctoral positions at Charles University. Most of the research in this paper forms part of N.B.'s bachelor thesis~\cite{Naty-thesis}, written at Charles University under S.G.'s cosupervision.

\section{AGM over $\R$ and $\C$ and their applications.}\label{sec:applications}

\begin{definition}\label{defAGM_R}
    Let $a, b$ be positive real numbers. The sequence $$AGM_\R(a,b)=((a_n,b_n))_{n=0}^\infty$$ is defined as follows. Let $a_0 = a$, $b_0 = b$ and for $n>0$, we define 
    \begin{align*}
        a_n &:= \frac{a_{n-1}+b_{n-1}}{2}, & b_n:&=\sqrt{a_{n-1}b_{n-1}}.
    \end{align*}

\end{definition}

As it is clear from Definition~\ref{defAGM_R} that $(a_n)_{n\in\N}$ and $(b_n)_{n\in\N}$ have the same limit, we denote $$M(a,b)\colonequals \lim_{n\rightarrow \infty} a_n =  \lim_{n\rightarrow \infty} b_n.$$

AGM sequences have applications in mathematical analysis, geometry, number theory, etc. Here, we explain some of these applications without going into too many details. Perhaps the most famous application is that one can use $AGM_\R$ for rapidly computing digits of $\pi$. The following algorithm was based on Gauss's work and was independently discovered by Brent and Salamin; see \cite[~p.48-49]{borwein1987pi} for more details. Gauss, in 1799, already noticed that AGM sequences are connected with elliptic integrals of the first kind; see more details in \cite[Theorem 1.1]{borwein1987pi}. We have
\[
\dfrac{1}{M(1,\sqrt{2})}=\dfrac{2}{\pi}\int_0^1\dfrac{dx}{\sqrt{1-x^2}}=\dfrac{2}{\pi}\int_0^{\frac{\pi}{2}}\dfrac{d\theta}{\sqrt{1+\sin^2\theta}}.
\]
In fact, it is interesting to note that elliptic integrals are invariant on the AGM sequence. For
\[
T(a,b)=\dfrac{2}{\pi}\int_0^{\frac{\pi}{2}}\dfrac{d\theta}{\sqrt{a^2\cos^2\theta+b^2\sin^2\theta}}
\]
and for $((a_n,b_n))_{n=0}^\infty=AGM_\R(a,b)$, we have $T(a,b)=T(a_n,b_n)$ for each $n$.

Further, one can use these results to compute that the length of the lemniscate given in the polar coordinates as $$r^2=\cos 2\theta$$ is $$\frac{2\pi}{M(1,\sqrt{2})}\approx5.244;$$ more details can be found at \cite[p.280]{Cox-AGM}. Also, we recommend \cite{borwein1987pi} and \cite{Cox-AGM} for finding more material on the AGM sequences and their applications.

Finally, we note that one can compute the inverse tangent function using AGM sequences in the following way (see \cite[p. 6-10]{Acton-arctg})
\[\tan^{-1}x=\lim_{n\rightarrow \infty}\dfrac{x}{a_n\sqrt{1+x^2}}, \; \text{where}\; ((a_n,b_n))_{n=0}^\infty=AGM_\R\left(\dfrac{1}{\sqrt{1+x^2}},1\right).\]

The AGM sequence can be defined over $\C$ as well. It is unclear which square root to take for the next term, and there are many choices (two in each step, hence, uncountably many!). However, all but countably many such sequences satisfy that both coordinates converge to 0. To avoid this, one introduces a condition for the starting pair $(a,b)$; see, e.g., \cite[Section 2, Proposition 2.1]{Cox-AGM}. The arithmetic-geometric mean $M(a,b)$ can be regarded as a multiple-valued function of $a$ and $b$. 

We briefly mention applications in number theory. The AGM sequences can be connected to the theta functions, hence to modular forms, see, e.g., \cite[Sections 2, 3]{borwein1987pi} or \cite{Cox-AGM}. Hence, it is not surprising that AGM sequences can be applied to the arithmetic of elliptic curves, for instance, for computing canonical heights on them (e.g., see \cite{Cremona-Thongjunthug} and the PhD thesis of Thongjunthug \cite{Thongjunthug-PhD}) or to compute or for point counting and computing the Serre-Tate lift (e.g., see the PhD thesis of Carls \cite{Carls-PhD}). So, the connection between elliptic curves and AGM sequences could serve as a motivation in \cite{griffin2022agm} and \cite{NewAGM} for the authors to use elliptic curves to study the AGM sequences over finite fields.

\section{AGM as a directed graph over finite fields.}\label{secAGMGraph}

In the whole paper, let $q$ be an odd prime power and $\F_q$ the field with exactly $q$ elements. We briefly recall some well-known properties of $\F_q$ that we will use many times.

\begin{definition}
    Let $x \in \F_q^\times$, then $x$ is a quadratic residue if there exists $y \in \F_q$ such that 
    \[y^2 = x.\]
    If there is not such $y$, $x$ is a quadratic non-residue. We will also call quadratic residues squares and non-residues non-squares.
\end{definition}

\begin{definition}
    Let $\phi_q \colon \F_q^\times \to \{\pm1\}$ such that for every $a \in \F_q$ the following holds: 
    \[
    \phi_q(a) =
    \begin{cases}
    1 & \text{if $a$ is a quadratic residue};\\
    -1 & \text{if $a$ is a quadratic non-residue}.
\end{cases}
    \]
\end{definition}

It is well-known that $\phi_q \colon \F_q^\times \to \{\pm1\}$ is a homomorphism.\\

We use some standard notions from the graph theory, which can be found, e.g., in \cite[Chapters 1, 10]{graphsbook}. We just recall one definition.

\begin{definition}
    Let $u,v$ be vertices in a directed graph. Then we will call $u$ a parent of $v$ if there is an edge $u \rightarrow v$. We will call $u$ a child of $v$ if there is an edge $v \rightarrow u$.
\end{definition}

Now, we are able to define a directed graph which will represent the AGM sequences over $\F_q$. 

\begin{definition}\label{defJfq}
Define a directed graph $\J_{\F_q}=(V,E)$ as follows: $V = \{(a,b) \in {\F_q^\times}^2 \mid \phi_q(ab)=1, a \neq \pm b\}$ and $(a,b) \rightarrow (c,d)$ is an edge if and only if 
    \begin{align*}
        c &= \frac{a+b}{2} & d^2 &=ab.
    \end{align*}
We will denote the components of $\J_{\F_q}$ as $J_i$ and the number of components $d(\F_q)$.
\end{definition}

Even though the following lemma is mostly contained in \cite[Theorem 1 (3)]{griffin2022agm}, for the convenience of the reader, we state it here with proof and emphasize that it does not depend on $q$ modulo 4.

\begin{lemma}\label{lemParents}
    Let $(a,b)$ be a vertex of $\jf$. Then $(a,b)$ has a parent if and only if $\phi_q(a^2-b^2)=1$. Furthermore, if $(a,b)$ has a parent, there are exactly two parents, namely, $(a + S, a - S)$ and $(a-S,a+S)$, where $S^2 = a^2-b^2$.
\end{lemma}
\begin{proof}
Suppose that $(a,b)$ has a parent $(A,B)$. We have $A+B=2a$ and $AB = b^2$, so $A$ and $B$ are roots of the quadratic polynomial $x^2-2ax + b^2 = (x-A)(x-B)$. Hence, the only two parents of $(a,b)$ are exactly $(A,B)$ and $(B,A)$. Furthermore, $\phi_q(a^2-b^2) = \phi_q(\frac{1}{4}(A-B)^2)=1$. 
    
On the other hand, if $\phi_q(a^2-b^2)=1$, there exist $S \in \F_q^{\times}$ such that $S^2 = a^2-b^2 \neq 0$. 
Then, we can consider the vertex $(a+S,a-S)$ as $a+S \neq \pm (a-S)$ and 
    \[\phi_q\left((a+S)(a-S)\right) = \phi_q(a^2-S^2) = \phi_q(a^2 - (a^2-b^2)) = \phi_q(b^2) = 1.\]
Since we have
    \[\frac{(a+S)+(a-S)}{2} = a\;\;\text{and}\;\;(a+S)(a-S) = b^2,\]
it follows that $(a+S, a-S)$ is a parent of $(a,b)$ and similarly for $(a-S, a+S)$.
\end{proof}

The graph $\J_{\F_q}$ has a group of natural automorphisms, as proven in \cite[Theorem 1(4)]{griffin2022agm}.

\begin{lemma}\label{lemmagroupautomorphisms}
Every $\alpha \in \F_q^\times$ induces a distinct graph automorphism 
\begin{align*}
    \varphi_\alpha \colon \J_{\F_q} &\to \J_{\F_q}\\
    (a,b) & \mapsto (\alpha a,\alpha b).\\
\end{align*}
Furthermore, $G=\{\varphi_\alpha \mid \alpha \in \F_q^\times\} \simeq \F_q^\times$ is a group under the obvious group law.
\end{lemma}

We look at the components of $\J_{\F_q}$. Let $N_n$ denote the number of components of $\jf$ with $n$ vertices. Then, \cite[Theorem 1(4)]{griffin2022agm} states that $(q-1)|nN_n$. However, we can prove a slightly better result. 
Namely, for any positive integer $n$, denote by $M_n$ the number of oriented cycles of the length $n$. When $q\equiv 3\pmod{4}$, \cite[Theorem 1(3,4)]{griffin2022agm} implies that in this situation we have  $nM_n=\frac{1}{2}nN_n$.

\begin{theorem}\label{theoremNumberOfComponents}
For any positive integer $n$, we have $(q-1)|nM_n$.
\end{theorem}

\begin{proof}
     Consider $G$ from Lemma \ref{lemmagroupautomorphisms}, the group of graph automorphisms $\varphi_\alpha$ such that $\alpha \in \F_q^\times$. This group acts on the $\J_{\F_q}$. Note that every orbit of a vertex has size $q-1$ as $(\alpha a,\alpha b)\neq (\beta a,\beta b)$ for $\alpha\neq\beta$ and $(a,b)\in V$. Furthermore, a vertex from a cycle is mapped to a vertex from a cycle of the same size. This implies that $q-1$ divides the number of all vertices which are in the cycles of the same length, so $q-1 \mid nM_n$.

\end{proof}

We can make one more observation which has not been mentioned in \cite{griffin2022agm}. 

\begin{definition}\label{def-square/non-square-vertices}
We call the vertex $(a,b)\in V$ a square vertex if $a$ and $b$ are both squares. The other vertices in $V$ are called non-square vertices.    
\end{definition}

\begin{theorem}\label{theoremSquareandNonsquareVertices}
\hfill
\begin{itemize}
\item If the edge in the component connects two square vertices, then all the vertices in the component are square vertices.
\item If the edge in the component connects two non-square vertices, then all the vertices in the component are non-square.
\item If the edge in the component connects a square vertex and a non-square vertex, then in the component, the square and non-square vertices alternate, so there is no edge in this component connecting two square vertices or two non-square vertices. 
\end{itemize}
\end{theorem}

\begin{proof}
The proof will follow from the following lemma.

\begin{lemma}\label{lemma-two-steps}
Consider the path $(a,b) \rightarrow (c,d) \rightarrow (e,f)$. Then $(a,b)$ is a square vertex if and only if $(e,f)$ is a square vertex.    
\end{lemma}

\begin{proof}
Suppose $(a,b)$ is a square vertex, so $a = A^2$ and $b = B^2$. Then 
    \begin{align*}
        c &= \frac{A^2+B^2}{2} & d^2 &= A^2B^2,
    \end{align*}
    which means that $d =\pm AB$, implying
 \[e = \frac{c+d}{2}= \left(\frac{A\pm B}{2}\right)^2. \]
 Then, by Definition~\ref{defJfq}, $f$ is a square.   

Now, assume that $(a,b)$ is a non-square vertex. Let $\alpha \in \F_q^\times$ be a non-square and consider the graph automorphism $\varphi_\alpha((a,b))=(\alpha a,\alpha b)$. Then
    \begin{align*}
        (a,b) \rightarrow (c,d) \rightarrow (e,f) &&\mapsto&& \varphi_\alpha((a,b)) \rightarrow \varphi_\alpha((c,d)) \rightarrow \varphi_\alpha((e,f)).
    \end{align*}
    Since $\phi_q(\alpha a) = (-1)(-1) = 1$,  $\varphi_\alpha((a,b))$ is a square vertex. We have already proved that $\varphi_\alpha((e,f))$ is a square vertex and as $\varphi_\alpha((e,f)) = (\alpha e, \alpha f)$, we have $\phi_q(e) = -1$, which completes the proof of the lemma.
\end{proof}

Hence, if we have a path between two vertices that are both squares or non-squares, then all vertices are of the same time; otherwise, they will alternate.
\end{proof}
\begin{corollary}\label{corollaryOddCycles}
Let $n$ be an odd positive integer. Then, $M_n$ is even. 
\end{corollary}

\begin{proof}
It follows directly from Theorem~\ref{theoremNumberOfComponents} as $q-1$ is even. But now we have a new insight after Theorem~\ref{theoremSquareandNonsquareVertices}. Namely, such directed circles cannot have alternating vertices. Hence, all of their vertices have to be squares or non-squares, and there is an equal number of them. 
\end{proof}

\section{AGM over $\F_q$, where ${q \equiv 3 \pmod 4}$.}\label{chap3mod4}

We now recall the results from \cite{griffin2022agm}, which served as an inspiration for our work. In the whole section, we assume $q \equiv 3 \pmod 4$. Then, we know that -1 is a quadratic non-residue in $\F_q$.

\begin{lemma}\label{lemmaEpsilon}
    Let $\varepsilon \in \{\pm 1\}$, $x \in \F_q^\times$. If $\phi_q(x) = 1$, then there is a unique $y \in \F_q^\times$ such that $y^2 = x$ and $\phi_q(y) = \varepsilon$. We will denote $\sqrtf{x} = y$.
\end{lemma}
\begin{proof}
    Let $t\in \F_q^\times$ be so that $t^2=x$. Then $\phi_q(t)=-\phi_q(-t)$, so either $\phi_q(t)=\varepsilon$ or $\phi_q(-t)=\varepsilon$. 
\end{proof}

\begin{definition}\label{defAGMFq}
    Let $a,b \in \F_q^\times$ such that $\phi_q(ab)=1$ and $a \neq \pm b$. We define a sequence $AGM_{\F_q}(a,b) = ((a_n,b_n))_{n=0}^\infty$ such that $a_0 = a$, $b_0 = b$ and 
    \begin{align*}
        a_n &= \frac{a_{n-1}+b_{n-1}}{2}; & \varepsilon &= \phi_q(a_n), & b_n &= \sqrtf{a_{n-1}b_{n-1}}.
    \end{align*}
\end{definition}

\begin{example}
    Let $q=3$. We have $\F_q^\times = \{1,-1\}$ so there is no $(a,b)$ such that $a\neq \pm b$.
\end{example}

For $q\geq 7$,  $AGM_{\F_q}(\cdot,\cdot)$ is an infinite sequence $(a_n,b_n)_{n\in \N} \in {\F_q^\times}^2$, such that $a_n b_n(a_n^2-b_n^2)\neq 0$.

\begin{example}
    Consider $q=11$. Then,
    \begin{align*}
        \agmel(4,1)=(&\overline{(4,1),(8,2),(5,4), (10, 8), (9, 5), (7, 10), (3, 9), (6, 7), (1, 3), (2, 6)})\\
        \\
        \agmel(1,4)=(&(1,4),\overline{(8,2),(5,4), (10, 8), (9, 5), (7, 10), (3, 9), (6,7), (1,3), (2, 6), (4,1)})\\
        \\
        \agmel(9,1)= (&\overline{(9, 1), (5, 3), (4, 9), (1, 5), (3, 4)})\\
        \\
        \agmel(1,9) = (&(1,9),\overline{(5, 3), (4, 9), (1, 5), (3, 4), (9, 1)}),
    \end{align*}
    where the line above a sequence means it is a repeating period.
\end{example}

\begin{example}
    For $q=11$, $\J_{\F_q}$ consists of 3 connected components given above.
\begin{center}
\begin{figure}[!h]
    \begin{center}
        \includegraphics[width = 7.5cm]{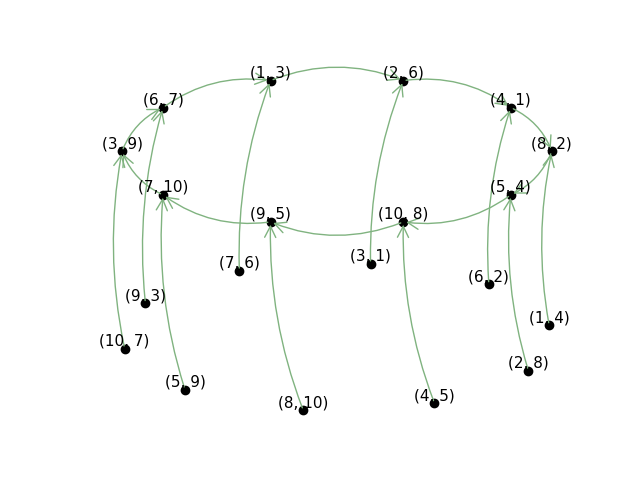}
    \end{center}
    \includegraphics[width = 6cm]{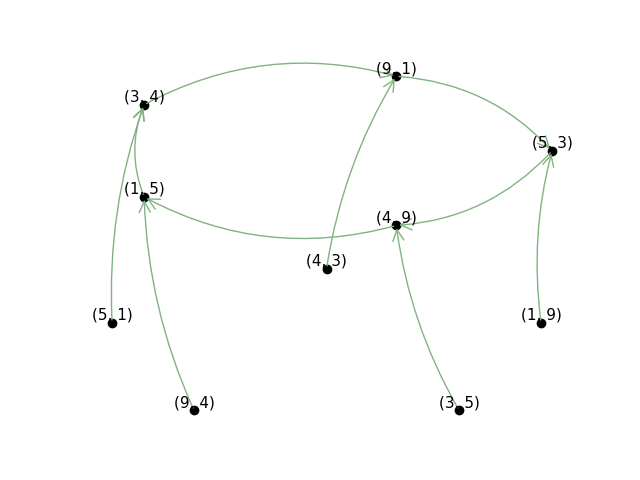}
    \includegraphics[width = 6cm]{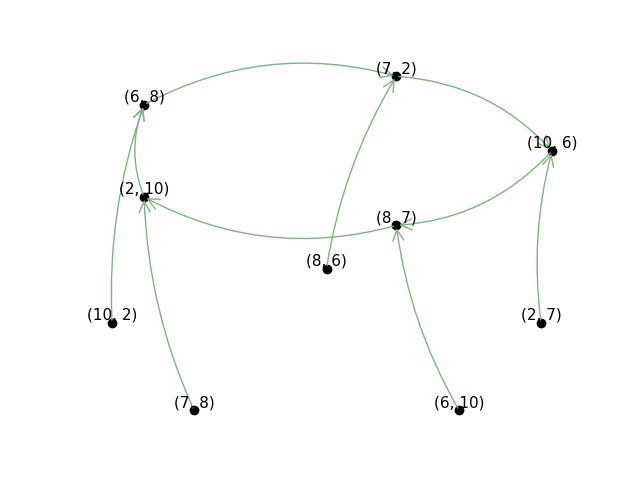}
    \caption{Components of $\J_{\F_{11}}$}
\end{figure}
\end{center}
\end{example}

Now, let us have a look at the graph of $\jf$. We will see that the components have a special shape, which is not a coincidence, as it is proven in \cite[Theorem 1]{griffin2022agm}.

\begin{theorem}\label{theoremShape3mod4}\hfill
    \begin{enumerate}[(1)]
    \item $\J_{\F_q}$ has $(q-1)(q-3)/2$ vertices.
    \item Every component of $\J_{\F_q}$ consists of one directed cycle, and there is a directed path of the length 1 to every vertex of that cycle. There are no other cycles (even undirected).
    \end{enumerate}
\end{theorem}

When we draw the component of $\jf$ in the plane, it looks like a bell head with tentacles. Hence
the components of $\jf$ are playfully called jellyfishes and the graph $\jf$ is a jellyfish swarm in \cite{griffin2022agm}.

\section{AGM over $\F_q$, where ${q\equiv~5~\pmod 8}$.}\label{chap5mod8} This section contains our new results, which are independent of those in \cite{NewAGM}, even though, not surprisingly, there are overlaps, especially with \cite[Section 4]{NewAGM}. 
In the whole section, let $q \equiv 5 \pmod 8$; then $-1$ is a square in $\F_q$, but 2 and $i$, such that $i^2=-1$ are not.

\begin{lemma}\label{lem:0-2-children}
Let $(a,b)$ be a vertex of $\jf$. Then $(a,b)$ has either $2$ or $0$ children.
\end{lemma}

\begin{proof}
Suppose $(a,b)$ has a child $(A,B)$, so $\phi_q(AB) = 1$. The equation $x^2=ab$ has exactly two solutions, $B$ and $-B$. Then, $\phi_q(A \cdot (-B))=1$ so, $(A,-B)$ is the other child. 
\end{proof}

Recall that $(a,b)$ has either $2$ or $0$ parents from Lemma ~\ref{lemParents}.

\begin{theorem}\label{theoremNonemptyEdges}
    Let $E$ be the set of edges of $\jf$. If $E \neq \emptyset$, then there is a cycle in $\jf$.
\end{theorem}
\begin{proof}
    Consider the edge $(a,b) \rightarrow (A,B)$ in $\jf$. Then the other child of $(a,b)$ is $(A,-B)$.
    Consider the pairs $(c,d)$ and $(e,f)$ (which do not have to be vertices of $\jf$) such that 
    \begin{align*}
        c &= \frac{A+B}{2} & d^2 &= AB\\
        e &= \frac{A-B}{2} & f^2 &= -AB.
    \end{align*}
    Then $d = \pm if$, where $i \in \F_q$ is such that $i^2=-1$, and we compute 
$$\phi_q(cdef) = \phi_q(ce) \phi_q(\pm if^2) = \phi_q \left(\frac{A^2-B^2}{4} \right) \cdot \phi_q(\pm 1)\phi_q(i)\phi_q(f^2)= \phi_q \left( \left(\frac{a-b}{4}\right)^2 \right) \cdot (-1) = -1.$$

As $cdef$ is a quadratic non-residue, exactly one of $cd$ and $ef$ is a quadratic residue, so exactly one of $(c,d)$, $(e,f)$ is a vertex. So, we can make the path longer and continue analogously for either $(A,B)$ or $(A,-B)$, but not both. As we can continue this way in every step, and the number of vertices is finite,  there must be a cycle in $\jf$. 
\end{proof}

\begin{corollary}\label{cor-two-children-one-in-cycle}
Let $C \subseteq V$ be the set of vertices in a cycle. Then, every $c\in C$ has exactly two children; one is a part of the cycle, and the other does not have any children.
\end{corollary}
\begin{proof}
Suppose $c_1 \in C$ has exactly two children, $c_2$ and $d_2$. Suppose $c_2\in C$, so it has exactly two children, $c_3$ and $c_3'$. Hence, by the proof of  Theorem~\ref{theoremNonemptyEdges}, then $d_2$ cannot have any children.
\end{proof}

\begin{theorem}\label{theoremMorethan2}
     Let $C$ be the cycle in $\jf$, then there are more than two vertices in $C$.
\end{theorem}

\begin{proof}
     Suppose that there are edges $(a,b) \rightarrow (c,d)$ and $(c,d) \rightarrow (a,b)$ in $\jf$. 
    Then,
$$ c = \frac{a+b}{2} \;\; \text{and} \;\; a= \frac{c+d}{2} \;\; \text{imply} \;\; a+c = b+d.$$
    
Also, $d^2=ab$ and $b^2=cd$ implies $ac = bd$.
    
By Vieta's formulas,  $(x-a)(x-c) = (x-b)(x-d)\in \F_q[x]$, so $\{a,b\}=\{c,d\}$.

    Recall that by Definition~\ref{defJfq}, we have $a\neq b$, hence, suppose $a=d, b = c$.
    So, we have an edge $(a,b) \rightarrow (b,a)$ which gives us $b = \frac{a+b}{2}$, thus $b = a$, a contradiction.
\end{proof}

\begin{theorem}\label{theoremLongerparents}
    Consider the path $(a,b) \rightarrow (c,d) \rightarrow (e,f)$ in $\jf$. Then, there is another path $(g,h) \rightarrow (d,c) \rightarrow (e,f)$ and exactly one of the vertices $(a,b),(g,h)$ has parents.
\end{theorem}
\begin{proof}
    By Lemma~\ref{lemParents}, we have $\phi_q(c^2-d^2)=1$, so $\phi_q(d^2-c^2)=1$. Then, by Lemma~\ref{lemParents}, there exist parents $(g,h)$ of $(d,c)$. Also, by Lemma~\ref{lemParents}, we know that there are $x,y\in \F_q^\times$ such that
     $a= c-x$, $b = c+x$, and $x^2= c^2-d^2$, and 
     $g= d-y$, $h = d+y$, and $y^2 = d^2 - c^2$.
    We can compute 
    \begin{align*}
        \phi_q((a^2-b^2)(g^2-h^2)) &= \phi_q\left(\left((c-x)^2-(c+x)^2\right)\left((d-y)^2-(d+y)^2\right)\right)\\
        &=\phi_q(16\cdot cx \cdot dy) = \phi_q(16) \phi_q(cd) \phi_q(xy) = \phi_q(xy)= -1,
    \end{align*}
because $y=\pm ix$. Thus, either $(a^2-b^2)$ or $(g^2-h^2)$ is a quadratic residue. By Lemma~\ref{lemParents}, either $(a,b)$ or $(g,h)$ has parents, and the other does not.
\end{proof}

\begin{corollary}\label{cor-two-parents-one-in-cycle}
    Let $C \subseteq V$ be the set of vertices in a cycle. Then every $c\in C$ has exactly 2 parents; one is part of the cycle $C$, and the other is not part of any cycle.
\end{corollary}

\begin{proof}
Suppose $c \in C$ and recall by Theorem~\ref{theoremMorethan2} that the length of a cycle is at least $3$, so we can consider the path $c_1 \rightarrow c_2 \rightarrow c$ of vertices from the cycle.
As $c$ has another parent ${c_2}'$, by Theorem~\ref{theoremLongerparents} we know there is another path ${c_1}' \rightarrow {c_2}' \rightarrow c$, hence exactly one of the vertices $c_1$ and $c_1'$ has parents by Theorem~\ref{theoremLongerparents}. As $c_1 \in C$, it has parents, hence $c_1'$ does not have parents.
\end{proof}

When we join all the previous claims about $\jf$ together, we obtain the following theorem. It is useful to look at Figure 2 for the proof of Theorem~\ref{theoremGraph5mod8}.

\begin{theorem}\label{theoremGraph5mod8}
    Every component of $\jf$ is made either of single vertex or of the cycle $c_1 \rightarrow c_2 \rightarrow \dots \rightarrow c_n$, vertices $u_i$ such that $c_i \rightarrow u_{i+1}$, vertices $v_i$ such that $v_i \rightarrow c_i$ and $v_i \rightarrow u_i$ and vertices $w_i$ such that $w_i \rightarrow v_i$, $w_{n+i} \rightarrow v_i$ and vertices $x_i$ such that $w_i \rightarrow x_i$, $w_{n+i} \rightarrow x_i$. These are all edges and vertices of a nontrivial component of $\jf$.
\end{theorem}
\vspace*{-1.1cm}
\begin{center}
\begin{figure}[!h]
    \begin{center}
        \includegraphics[width = 11.3cm]{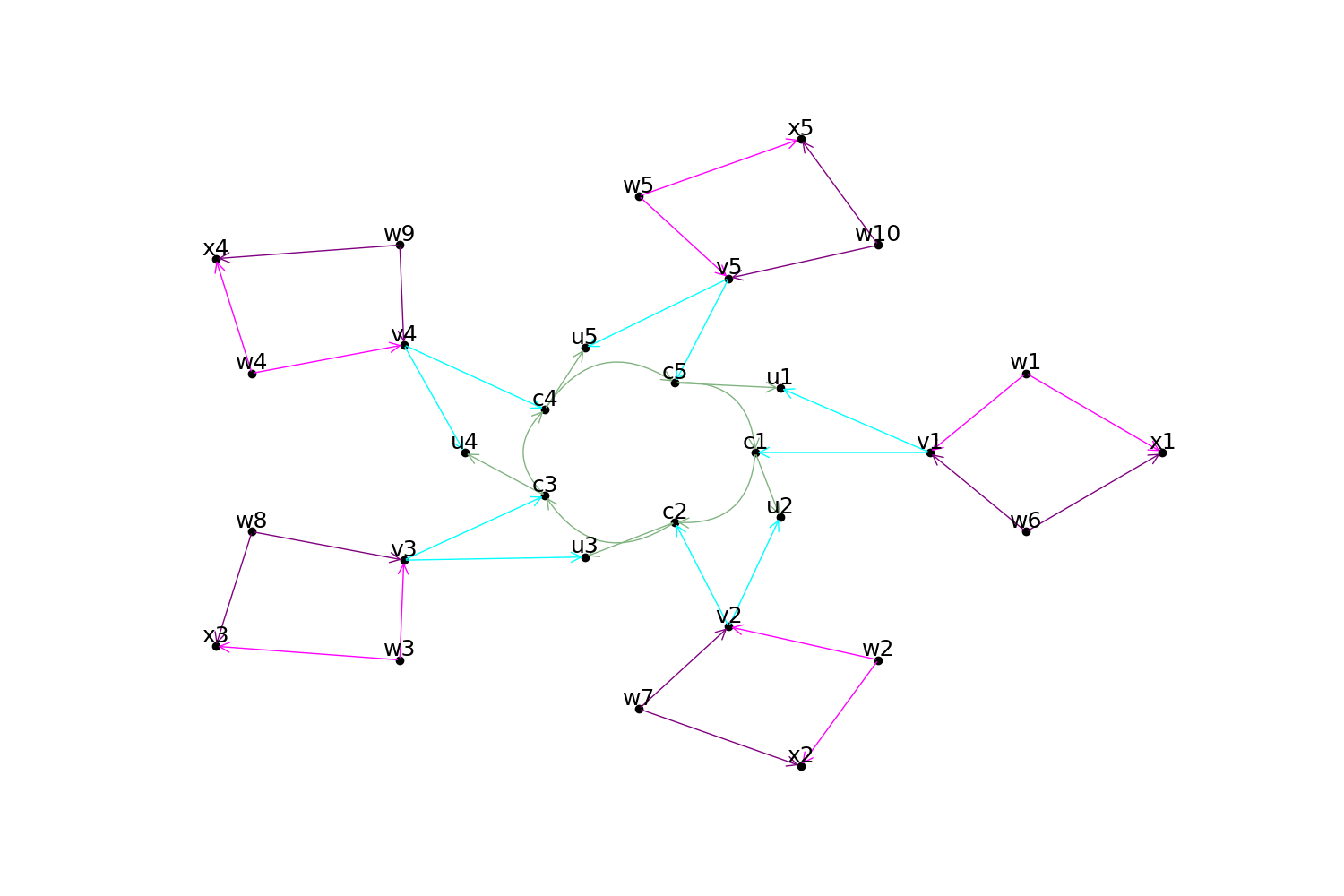}
    \end{center}
    \caption{A general nontrivial component of $\jf$ such that $q \equiv 5 \pmod 8$}
\end{figure}
\end{center}

\begin{proof}
Assume there is a cycle $C$ of length $n$ (otherwise, by Theorem~\ref{theoremNonemptyEdges}, only vertices are the components). Consider the path $c_n=(a_n,b_n)\rightarrow c_1=(a_1,b_1)\rightarrow c_2=(a_2,b_2)$. By Corollary~\ref{cor-two-children-one-in-cycle}, $c_1$ has two children, one is $c_2=(a_2,b_2)$ in the cycle, and the other one is $u_2=(a_2,-b_2)$, which is not in the cycle and has no children. We also know that there is another parent of $c_2$ and $u_2$, which is $v_2=(b_1,a_1)$. Since $c_1=(a_1,b_1)$ has parents because $c_1\in C$, this implies that $\phi_q(a_1^2-b_1^2)=1$, and so $\phi_q(b_1^2-a_1^2)=1$, thus $v_2=(b_1,a_1)$ also has parents by Lemma~\ref{lemParents}. Let these parents be $w_2=(e_1,f_1)$ and $w_{2+n}=(f_1,e_1)$, and they have another child $x_2=(b_1,-a_1)$. Since $v_2$ has children, we know that $x_2$ has no children. Now, we apply Theorem~\ref{theoremLongerparents} to the paths $c_n\rightarrow c_1\rightarrow c_2$ and $w_2\rightarrow v_2\rightarrow c_2$ (or $w_{n+2}\rightarrow v_2\rightarrow c_2$) to conclude that $w_2$ and $w_{n+2}$ have no parents because $c_n\in C$, so it has parents. This completes the proof and explains the components containing a cycle in a bit more detail.
\end{proof}

\begin{definition}\label{def:AGMq-5-mod8}
    Let $q \equiv 5 \pmod 8$  and $(a,b) \in {\F_q^\times}^2$ such that $ab(a^2-b^2)\neq 0$, $\phi_q(ab) = 1$ and there exist $c,d \in \F_q^\times$ such that $cd(c^2-d^2)\neq 0$ and
    \begin{align*}
        c &= \frac{a+b}{2}, & d^2&=ab, & \phi_q(cd) = 1.  
    \end{align*}
    Let $C$ be the component of $\jf$, which contains the vertex $(a,b)$, and suppose it contains a cycle of the length $n$. Then, using the notation from Theorem~\ref{theoremGraph5mod8}, without loss of generality, we may assume $(a,b) \in \{c_1, v_1, w_1\}$.
    We define an infinite arithmetic-geometric mean sequence as follows:
    \[\agmf(a,b) =
    \begin{cases}
        (\overline{c_1, c_2, \dots, c_n}) & \text{if } (a,b) = c_1,\\
        (v_1,\overline{c_1, c_2, \dots, c_n}) & \text{if } (a,b) = v_1,\\
        (w_1, v_1,\overline{c_1, c_2, \dots, c_n}) & \text{if } (a,b) = w_1.
    \end{cases}
    \]
\end{definition}
Recall that we denote the repeating period by line over elements of the sequence. 

\begin{theorem}\label{thm-nontrivial-graph-29}
    If $q\geq 29$ and $q \equiv 5 \pmod 8$, then there is at least one nontrivial component.
\end{theorem}
\begin{proof}
   We prove that there is a vertex $(a,b)$ in $\jf$ which has parents, i.e., with $\phi_q(a^2-b^2)=1$ by Lemma~\ref{lemParents}. Let $g$ be a generator of the group $\F_q^\times$. As $q\geq 29$, we have $g^2, g^4, g^6 \notin \{1,-1\}$, so we can consider vertices $(g^2,1)$, $(g^4,1)$, $(g^6,1)$. If any of 
    \begin{align*}
        (g^2)^2-1^2&=g^4-1, & (g^4)^2-1^2 &= g^8-1, & (g^6)^2-1^2 &= g^{12}-1
    \end{align*}
    is a square, and then we find a vertex with its parents. Suppose none of them is a square, then 
    \begin{align*}
        -1 &= \phi_q(g^8-1) = \phi_q((g^4-1)(g^4+1))
        = - \phi(g^4+1), \;\text{therefore},\; \phi(g^4+1)=1;\\
        -1 &= \phi_q(g^{12}-1) = \phi_q((g^4-1)(g^8+g^4+1)) =  - \phi_q(g^8+g^4+1), \;\text{therefore},\; \phi_q(g^8+g^4+1)=1.
    \end{align*}
    So, there are $a,b\in \F_q^{\times}$ such that $a^2 = g^4+1 \neq 0$ and $b^2 = g^8+g^4+1 \neq 0$. Then
\[a^4-b^2 = g^8+2g^4+1-(g^8+g^4+1) = g^4,\;\text{so},\; a^4-g^4 = b^2.\]
    We have $\phi_q(a^2g^2) = 1$ and $a^2g^2(a^4-g^4)\neq 0$, so, $(a^2,g^2)$ is a vertex of $\jf$ which has parents.
\end{proof}

We note that \cite[Theorem 3(1)]{NewAGM} is much stronger than this result for sufficiently large $q$, but we have the best possible bound for $q$ such that there is at least one nontrivial component in $\jf$. We can see that for smaller $q$, there are no nontrivial components. Further, we can be more precise than in Theorem~\ref{thm-nontrivial-graph-29}, and show that there is a nontrivial component consisting only of square vertices.

\begin{theorem}[Main theorem in \cite{sumofpowers}]\label{theoremksum}
Let $k$ be a positive integer, and $q>1$ be a prime power. Let $\delta = \gcd(q-1,k)$. Assume $q>(\delta -1)^4$. Then every element of $\F_q$ is a sum of two $k$th powers. 
\end{theorem}

\begin{corollary}\label{cor-all-squares-component}
If $q \equiv 5 \pmod 8$, $q\geq 29$, there is a nontrivial component in $\jf$ consisting only of square vertices. 
\end{corollary}
\begin{proof}
Let $g$ be a generator of $\F_q^{\times}$ and let $q>3^4$. We have $4 = (q-1,4)$ and from Theorem~\ref{theoremksum} we find $x$ and $y$ such that $x^4+y^4=2g^2$. If $x=\pm y$, then $x^4+y^4=2x^4=2g^2$, then $g=x^2$ or $g=(ix)^2$, which is a contradiction as $g$ is a generator of $\F_q^{\times}$. \\
So, we obtained an edge $(x^4,y^4) \rightarrow (g^2,x^2y^2)$. Both of the vertices are squares; hence, the component consists of square vertices by Theorem~\ref{theoremSquareandNonsquareVertices}.
For $q = 29, 37, 53, 61$, there is at least one square component, which we can compute by computer, for example, the components that contain cycles which contain
\begin{itemize}
\item $(9,6)$ for $q=29$;
\item $(27,4)$ for $q=37$;
\item $(25,7)$ for $q=53$;
\item $(1,9)$ for $q=61$.
\end{itemize}   
\end{proof}

\begin{theorem}\label{theoremMinimalField}
    There are $\frac{(q-1)(q-5)}{2}$ vertices in $\jf$. 
\end{theorem}
\begin{proof}
    There are $\frac{q-1}{2}$ squares which we can pair with $\frac{q-1}{2}-2=\frac{q-5}{2}$ squares (to avoid the pairs $(a,\pm a)$). The same is true for non-squares, so together, we have $\frac{(q-1)(q-5)}{2}$ vertices.
\end{proof}

\begin{example}
From Theorem~\ref{theoremMinimalField}, for $q=5$, there are no vertices.
\end{example}

\begin{example}
    Let $q = 13$. The vertices which contain $1$ as the first coordinate are $(1,3)$, $(1,4)$, $(1,9)$ and $(1,10)$,
    but none of them is part of an edge. Using symmetry and automorphisms $\varphi_\alpha$, we conclude that there are no edges in this graph.
\end{example}

\begin{example}
The smallest $q$ with nontrivial graph components is $q=29$. There is one cycle of length $28$ and $4$ cycles of length $7$.

\begin{figure}[!h]
    \centering
    \includegraphics[width = 15cm]{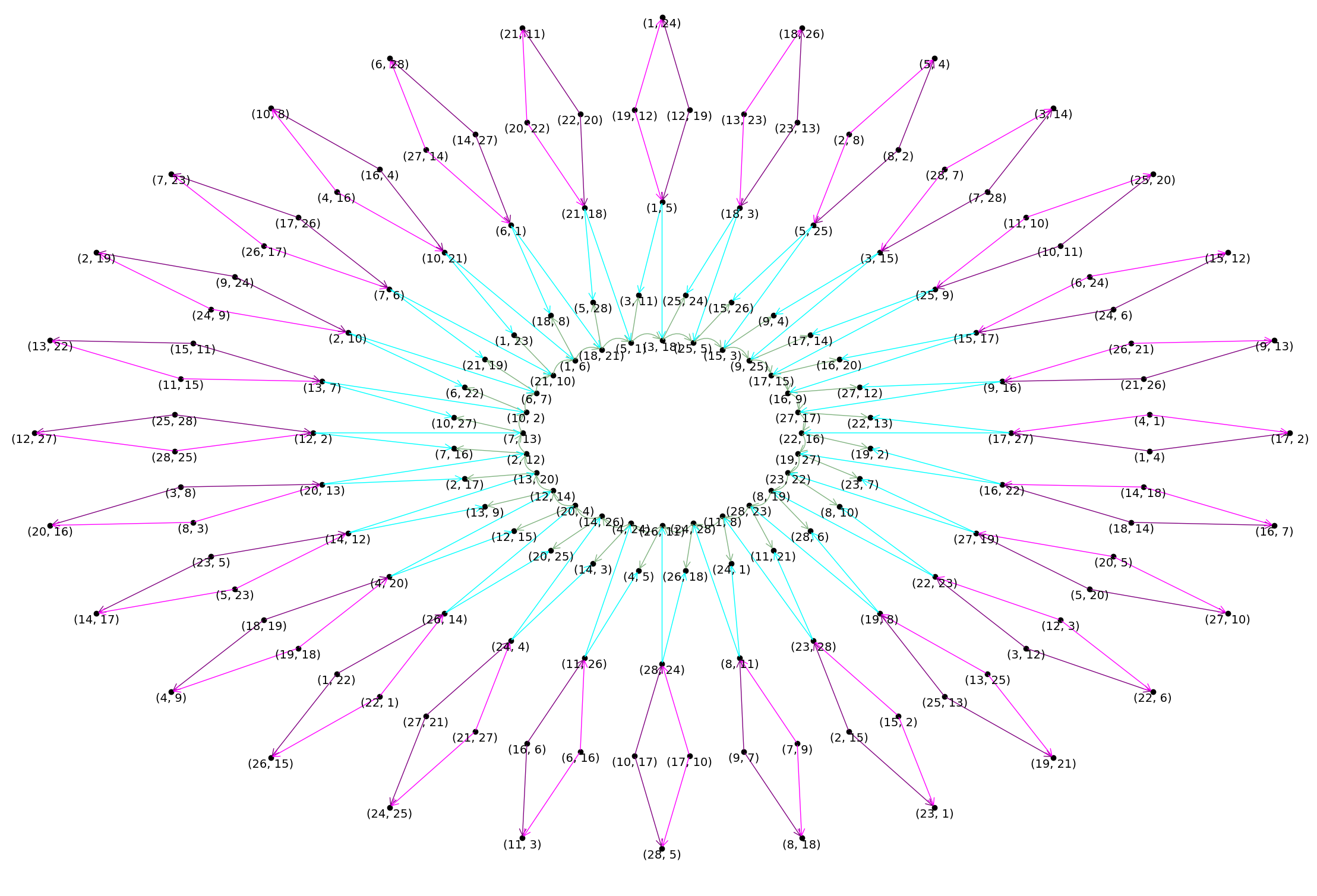}
\end{figure}
\begin{figure}[!h]
    \includegraphics[width = 7.2cm]{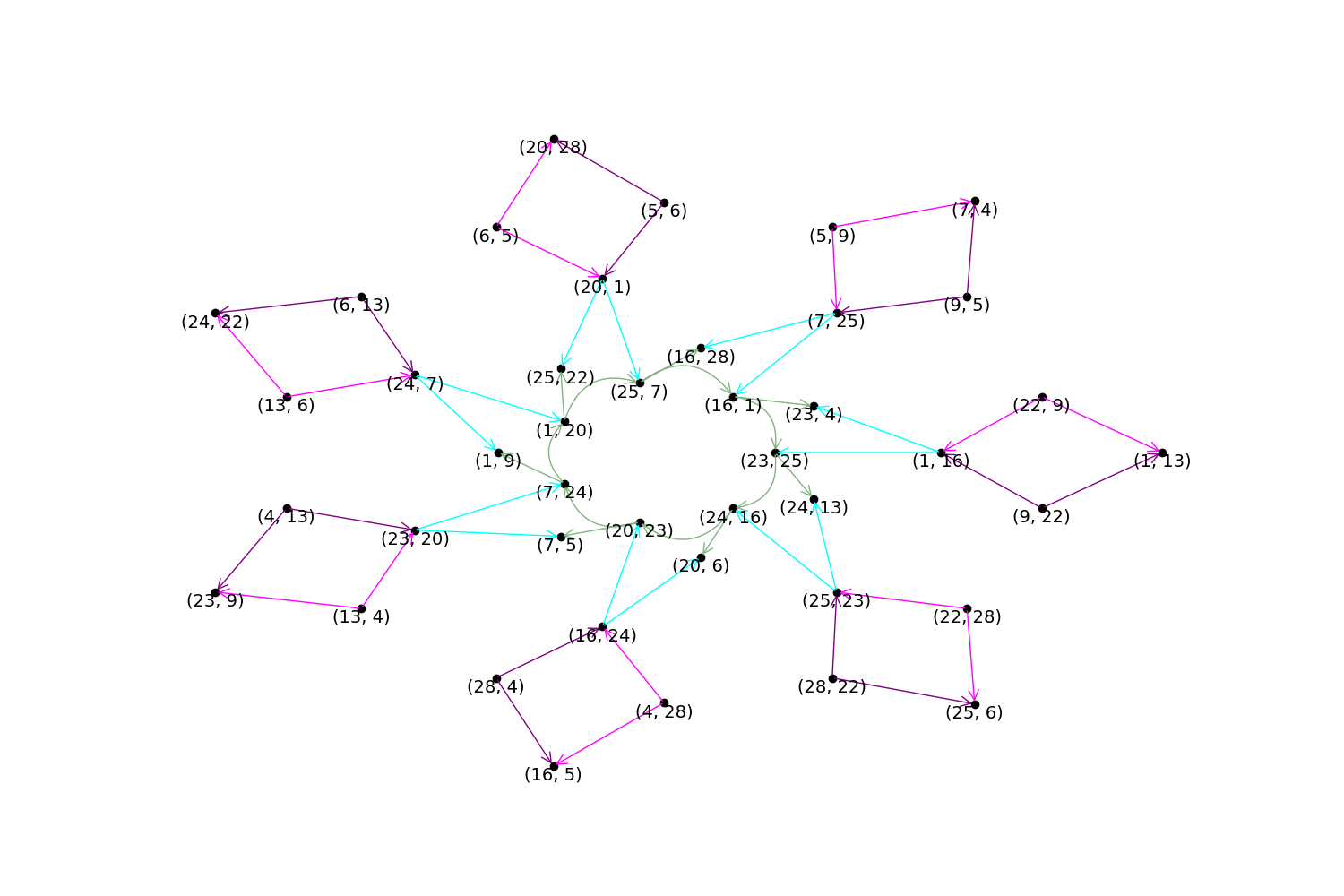}
       \includegraphics[width = 7.2cm]{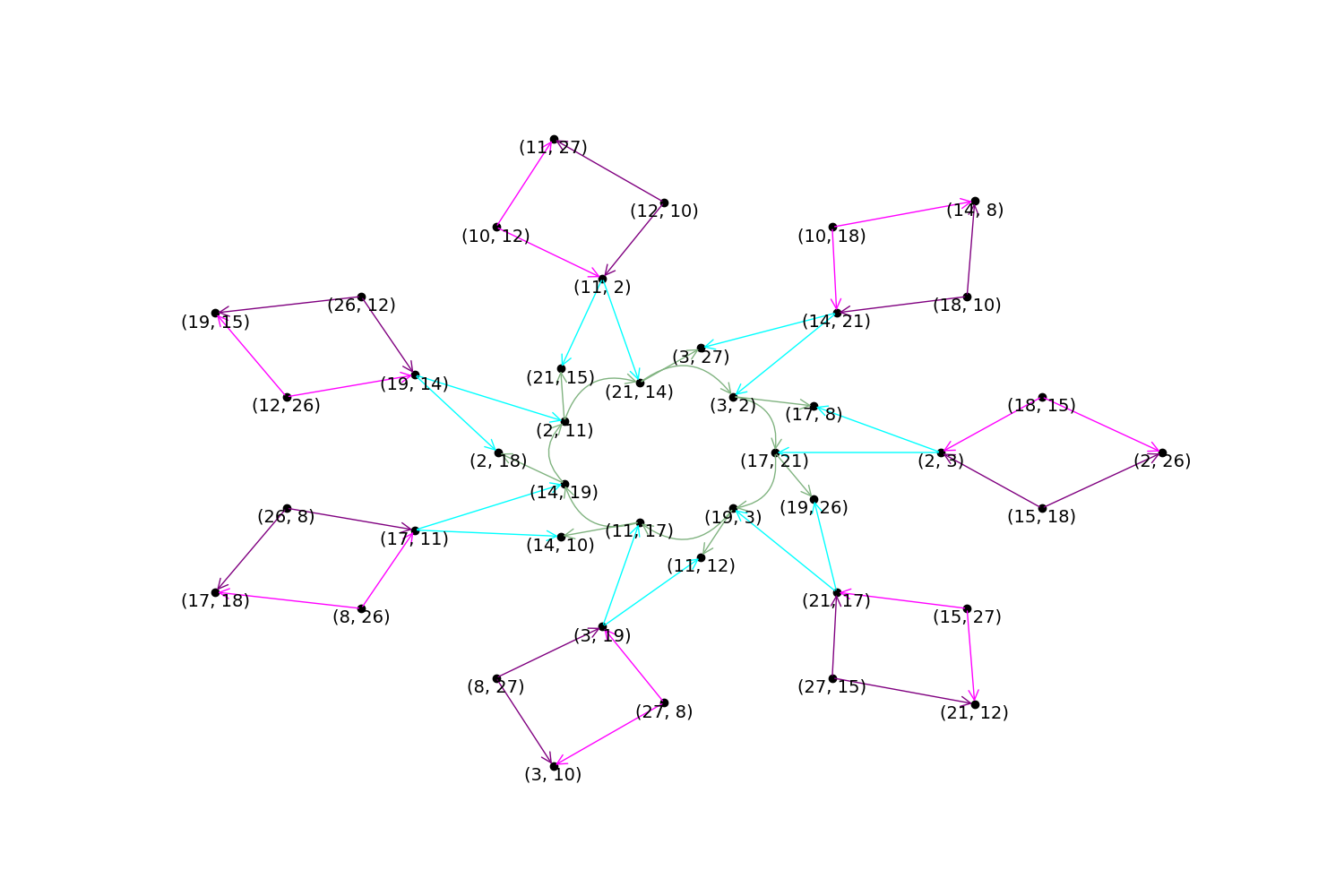}
       \end{figure}
       \begin{figure}[!h]
        \includegraphics[width = 7.2cm]{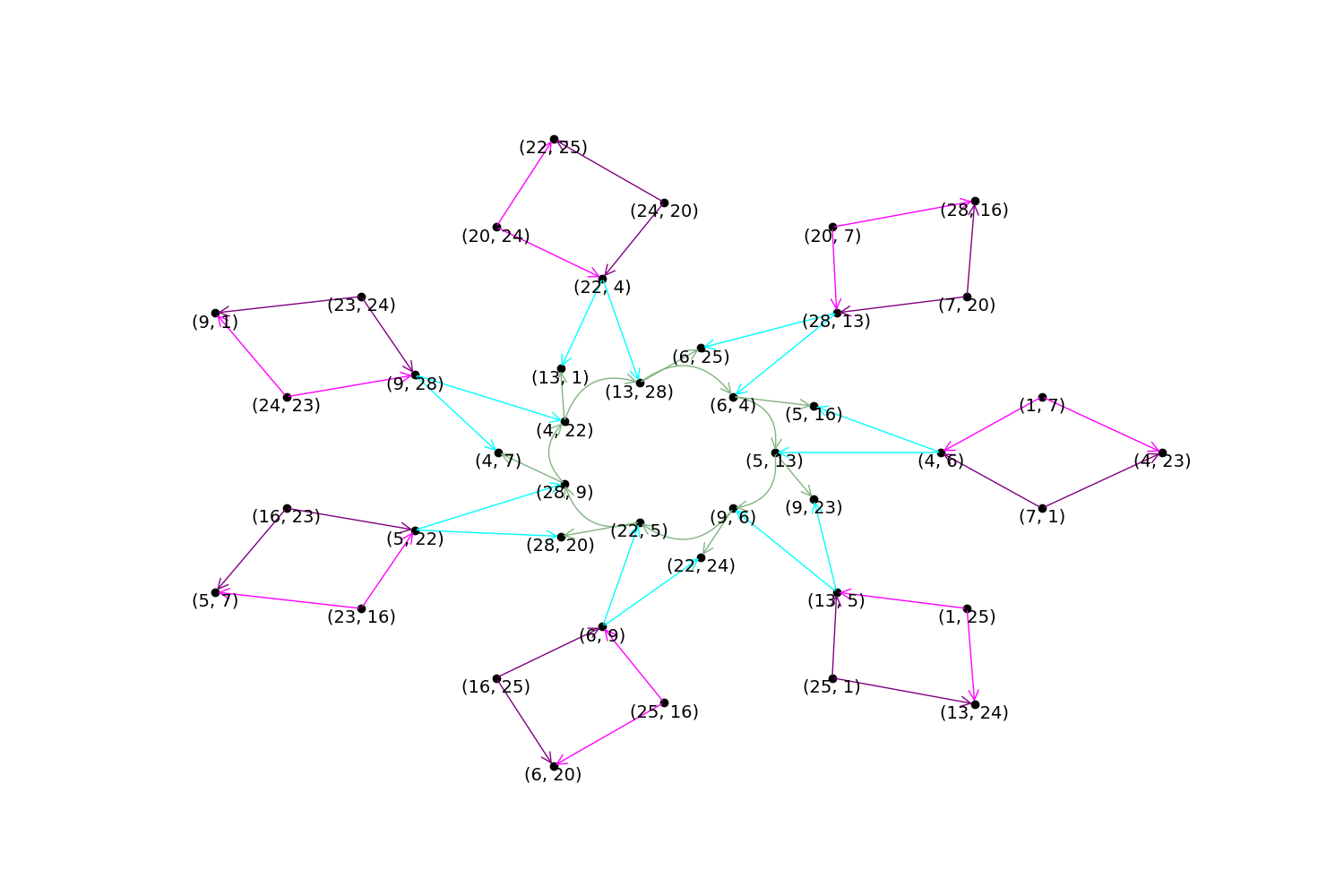}
        \includegraphics[width = 7.2cm]{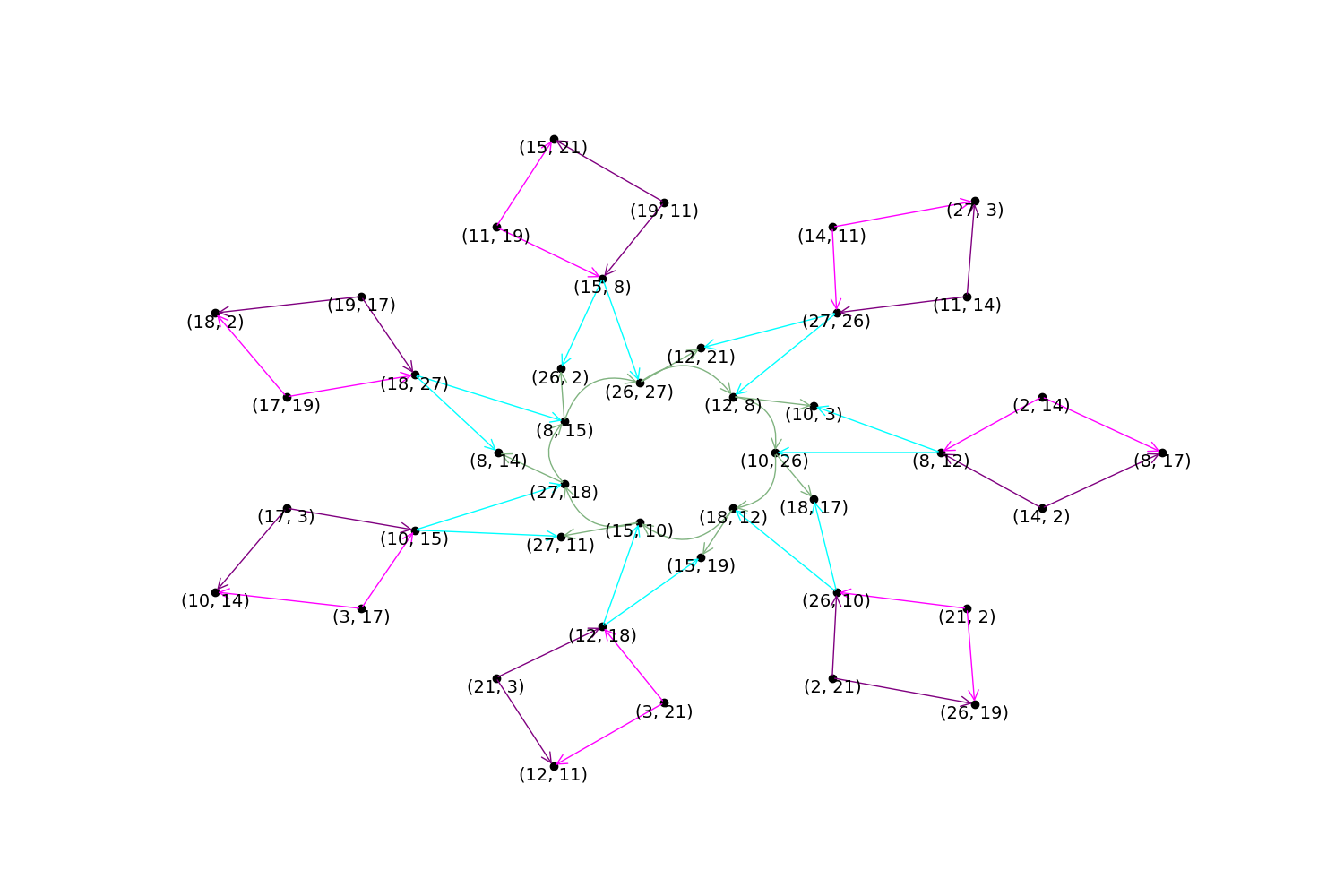}
    \caption{Components of $\J_{\F_{29}}$}
\end{figure}
\end{example}

\begin{section}{Challenge for $q\equiv 1\pmod{8}$.}\label{sec-challenge}
It might be tempting to try to generalise our approach to $q\equiv 1\pmod{8}$ in the following way.

For $q\equiv3\pmod{4}$, $-1$ is not a square in $\F_q$. If $(a_1,b_1)\in {\F_q^{\times}}^2$ is such that $\phi_q(a_1b_1)=1$, then 
\begin{equation*}
 \begin{tikzcd}
&  (a_1,b_1)  \arrow{ld} \arrow{rd}\\ 
  (a_2,b_2) &  & (a_2,-b_2).
\end{tikzcd}    
\end{equation*}
For exactly one of $(a_2,b_2)$ and $(a_2,-b_2)$, we have $\phi_q(\pm a_2b_2)=1$, and we use this pair for the definition of the sequence. This is the definition of the AGM over such $\F_q$ as in \cite{griffin2022agm}.\\

In this paper, we look at $q\equiv 5\pmod{8}$; then $-1$ is a square in $\F_q$ but $i\in \F_q$ such that $i^2=-1$, is not. Alternatively, the fourth root of unity is defined in $\F_q$, but not the eighth root of unity. We used this fact to define the AGM over such $\F_q$ whenever possible. Namely, assume that we have two steps, as in the following diagram

\begin{equation*}
\adjustbox{scale=0.8}{%
\begin{tikzcd}
& & & (a_1,b_1) \arrow{lld} \arrow{rrd}\\
& (a_2,b_2) \arrow{ld} \arrow{rd} & & & & (a_2,-b_2) \arrow{ld} \arrow{rd}\\
(a_3,b_3) & & (a_3,-b_3) & & (a_4,ib_3) & & (a_3,-ib_3). \\ 
\end{tikzcd} }   
\end{equation*}

We proved that $\phi_q(a_3b_3a_4ib_3)=-1$, and then we can continue with the AGM sequence in exactly one half of the diagram. But the key point here besides $\phi_q(i)=-1$ was that $\phi(a_3a_4)=1$ as this follows from Lemma~\ref{lemma-two-steps}, as $\phi_q(a_4)=\phi_q(a_3)=\phi_q(a_1)$.\\

Finally, let $q\equiv 1\pmod{8}$. Then, there is an eighth root of unity in $\F_q$, call it $\zeta_8$. If we furthermore assume that $q\equiv 9\pmod{16}$, then $\phi_q(\zeta_8)=-1$. We could try to look now at the first three steps of the AGM sequence:

\begin{equation*}
\adjustbox{scale=0.5}{%
\begin{tikzcd}
& & & & & & & (a_1,b_1) \arrow{lllld} \arrow{rrrrd}\\
& & & (a_2,b_2) \arrow{lld} \arrow{rrd} & & & & & & & & (a_2,-b_2) \arrow{lld} \arrow{rrd}\\
& (a_3,b_3) \arrow{ld} \arrow{rd} & & & & (a_3,-b_3) \arrow{ld} \arrow{rd} & & & & (a_4,ib_3) \arrow{ld} \arrow{rd} & & & & (a_3,-ib_3) \arrow{ld} \arrow{rd}  \\ 
(a_5,b_4) & & (a_5,-b_4) & & (a_6,ib_4) & & (a_6,-ib_4)  & &(a_7,b_5) & & (a_7,-b_5) & & (a_8,ib_5) & & (a_8,-ib_5). 
\end{tikzcd} }   
\end{equation*}

Here, again, if we know that we can continue for one element on the left or right half, we can continue for all of them. So, it suffices to consider pairs $(a_5,b_4)$ and $(a_7,b_5)$. Also, by Lemma~\ref{lemma-two-steps}, we have $\phi_q(a_5)=\phi_q(a_7)=\phi_q(a_2)$, so 
$\phi_q(a_5b_4a_7b_5)=\phi_q(b_4b_5)$. Unlike in the previous cases, it might not be true that $b_4b_5$ is a multiple of a square and $\zeta_8$. For $q\equiv 1\pmod{16}$, it is even less clear what to do up to this step. \\

Very recently, June Kayath, Connor Lane, Ben Neifeld, Tianyu Ni, Hui Xue, published a preprint \cite{NewAGM} in which they explain how to approach $q\equiv 1\pmod{8}$ in a clever way using elliptic curves. They achieve exciting results which extend \cite[Theorem 5]{griffin2022agm} about the number of components of the graphs corresponding to AGM sequences in cases $q\equiv 5\pmod{8}$ and $q\equiv1\pmod{8}$ (see \cite[Theorem 3]{NewAGM}).

We praise their results, and still, we leave a challenge for interested readers:\\

\noindent{\bf Question.} Is there an elementary way to prove that for sufficiently large $q\equiv 1\pmod{8}$, there is always a non-trivial component of the graph corresponding to the AGM sequence over $\F_q$?
\end{section}

\vfill\eject

\end{document}